\documentclass[anon,12pt]{colt2015} % Anonymized submission
\title[Exponential Family Matrix Completion]{Low Rank Matrix Completion with Exponential Family Noise}
\usepackage{times}
\usepackage{aliascnt}
\usepackage{amsopn,amssymb}
\usepackage{bbm,dsfont}
\usepackage{mathrsfs}
\usepackage{enumerate}
\usepackage[utf8]{inputenc}
\usepackage[textwidth=4cm, textsize=footnotesize]{todonotes}
\usepackage{color}

 \coltauthor{\Name{Jean Lafond} \Email{jean.lafond@telecom-paristech.fr}\\
 \addr Institut Mines-T\'el\'ecom,
T\'el\'ecom ParisTech, CNRS LTCI\\ 
 }
\usepackage{cleveref}
\newcommand{\ie}{{\em i.e.,~}}
\newcommand{\wlg}{{\em w.l.o.g.,~}}
\newcommand{\eg}{{\em e.g.,~}}

\newcommand{\resp}{{\em resp.~}}

\newcommand{\wrt}{{\em w.r.t.~}}
\def\RR{\mathbb{R}}
\def\EE{\mathbb{E}}
\def\PP{\mathbb{P}}

\def\iid{i.i.d.}
\def\eqs{\;}

\usepackage{color}

 \newtheorem{assumption}{H \hspace*{-5.5pt}}
\DeclareMathOperator{\tr}{tr}
\DeclareMathOperator{\rank}{rk}
\DeclareMathOperator{\diag}{diag}
\DeclareMathOperator{\Proj}{\mathcal{P}}
\DeclareMathOperator{\Lik}{\Phi_Y}
\DeclareMathOperator{\Likpi}{\Phi_Y^\Pi}

\def\speq{-}

\newcommand{\argmin}{\mathop{\mathrm{arg\,min}}}
\newcommand{\Breg}[3]{d_{ #1 }(#2,#3)}

\newcommand{\Bregemp}[2]{D^n_\Gexp(#1,#2)}

\newcommand{\Bregpi}[2]{D^\Pi_\Gexp(#1,#2)}

\newcommand{\mat}[1]{#1}

\newcommand{\matset}[2]{\RR^{#1 \times #2}}

\newcommand{\Obj}{\Phi_Y^{\lambda}}

\newcommand{\Bern}[1]{#1}
\newcommand{\Lc}{\nu}

\newcommand{\sexp}{\delta}
\newcommand{\Gexp}{G}
\newcommand{\gexp}{G'}
\newcommand{\dgexp}{G''}

\newcommand{\eqdef}{\mathrel{\mathop:}=}
\newcommand{\pscal}[2]{\langle#1\,|\,#2\rangle}
\newcommand{\Pscal}[2]{\left\langle#1\,|\,#2\right\rangle}
\def\sigup{\bar{\sigma}_\gamma}
\def\siglo{\underline{\sigma}_\gamma}

\def\schat_const{s}

\def\rset{\ensuremath{\mathbb{R}}}
\newcommand{\eqsp}{\;}
\newcommand{\empf}[2]{ \Delta^2_Y (#1, #2) }
\newcommand{\empfw}[2]{ \sum _{kl \in [m_1] \times [m_2]} \pi_{kl} (#1_{kl} - #2_{kl})^2 }

\newcommand{\KLd}[2]{{D}\left(#1\|#2\right)}
\newcommand{\rscset}[2]{ \mathcal{C}(#1,#2) }
\newcommand{\rscsett}[3]{ \mathcal{C}(#1,#2,#3) }

\def\tX{\bar{X}}
\def\est{\hat{\mat{X}}}
\def\gradpi{H}
\def\estpi{\check{\mat{X}}}
\def\bvarepsilon{\boldsymbol{\varepsilon}}

\newcommand{\Exp}[2]{\operatorname{Exp}_{#1,#2}}

\def\indic{\delta}
\def\rme{\mathrm{e}}
\newcommand{\norm}[1]{\left\Vert #1 \right\Vert}

\def\AA{{\cal A} }
\crefname{table}{Table}{Tables}
\Crefname{algocf}{Algorithm}{Algorithms}
 \crefname{proposition}{Proposition}{Propositions}
\newcommand*{\propositionrefname}{Proposition}
\newcommand*{\propositionsrefname}{Propositions}
\newcommand*{\propositionref}[1]{%
  \objectref{#1}{\propositionrefname}{\propositionsrefname}{}{}}

\begin{document}

\maketitle

\begin{abstract}
The matrix completion problem consists in
 reconstructing a matrix from a sample of
entries, possibly observed with noise.
A popular class of estimator, known as nuclear norm penalized estimators,
are based on minimizing the sum of a data fitting term and a nuclear norm
penalization.
Here, we investigate the case where
the noise distribution
belongs to the exponential
family and is sub-exponential.
Our framework alllows for a general sampling scheme.
We first consider
an estimator defined as the minimizer of the sum 
of a log-likelihood term and a nuclear norm penalization
and prove an upper bound
on the Frobenius prediction risk. The rate obtained improves on
previous works on matrix completion for exponential family.
When the sampling distribution is known,
we propose another estimator and prove an oracle inequality
\wrt the Kullback-Leibler prediction risk, which translates immediatly
into an upper bound on the Frobenius prediction risk. Finally, we show that all the rates obtained
are minimax optimal up to a logarithmic factor.
\end{abstract}


\maketitle

\section{Introduction}
In the matrix completion problem one aims at recovering
 a matrix, based on partial and noisy observations
of its entries. This problem arises in a wide range of practical
situations such as collaborative filtering
or quantum tomography (see \cite{Srebro_Salakhutdinov10} or \cite{Gross11} for instance).
In typical applications, the 
number of observations is usually much smaller than the total
number of entries, so that some structural constraints are
needed to recover the whole matrix efficiently.

More precisely, we consider an $m_1\!\times\!m_2$
real matrix $\tX$ and observe $n$ samples of the form $(Y_i,\omega_i)_{i=1}^n$,
with $(\omega_i)_{i=1}^n \in ([m_1]\!\times\![m_2])^n$  an $\iid$ sequence of indexes
and $(Y_i)_{i=1}^n \in \RR^n$ a sequence of observations which is assumed
to be $\iid$ conditionally to the entries $(\tX_{\omega_i})_{i=1}^n$.
To recover the unknown parameter matrix $\tX$, a popular class of methods,
known as penalized nuclear norm estimators, are based on minimizing
the sum of a data fitting term and a nuclear norm penalization term.
These estimators have been extensively studied over the past decade
and strong statistical
guarantees can be proved in some particular settings.
When the conditional distribution $Y_i|\tX_{\omega_i}$
is additive and sub-exponential it can be shown that the unknown matrix 
can be recovered efficiently,
provided that it is low rank or approximately low rank,
see \cite{Candes_Plan10,Keshavan_Montanari_Oh10,Koltchinskii_Tsybakov_Lounici11,Negahban_Wainwright12,Cai_Zhou13,Klopp14}.
In that case, the prediction error satisfies with high probability
\begin{equation}
\label{eq:rate}
 \frac{\|\est-\tX\|_{\sigma,2}^2}{m_1m_2} = \mathcal{O}\left(\frac{(m_1+m_2)\rank(\tX)\log(m_1+m_2)}{n}\right) \eqs,
\end{equation}
with $\est$ denoting the estimator, $\|\cdot\|_{\sigma,2}$ the Frobenius
norm and $\rank(\cdot)$ the rank of a matrix.
It has been proved by \cite{Koltchinskii_Tsybakov_Lounici11} that this rate is actually
minimax optimal up to a logarithmic factor.

Although very common in practice,
discrete distributions have received less attention. 
The analysis of a logistic noise was first addressed by
\cite{Davenport_Plan_VandenBerg_Wootters12}.
% \cite{Davenport_Plan_VandenBerg_Wootters14}. 
It was later considered by  \cite{Cai_Zhou14}, \cite{Lafond_Klopp_Moulines_Salmon14}
and \cite{Klopp_Lafond_Moulines_Salmon14}
who have shown that the prediction error is also of the order of \eqref{eq:rate},
for log-likelihood estimators, regularized with nuclear norm.
\cite{Gunasekar_Ravikumar_Gosh14} have investigated the case of
distributions belonging to the exponential family, which is rich enough to encompass
both continuous and discrete distributions
(Gaussian, exponential, Poisson, logistic, etc.).
They provide (see their Corollary 1) an upper bound for the prediction error when the noise is 
sub-Gaussian and the sampling uniform. However, this bound is of the form
\begin{equation*}
 \frac{\|\est-\tX\|_{\sigma,2}^2}{m_1m_2} = \mathcal{O}\left(\alpha^{*2}
 \frac{(m_1+m_2)\rank(\tX)\log(m_1+m_2)}{n}\right) \eqs,
\end{equation*}
where $\alpha^{*2}$ is of the order $m_1m_2$
(see \remarkref{rem:discus_rate} below for more details). Therefore,
the obtained rate does not match \eqref{eq:rate},
which suggests that there may have some room for improvement.

In the present work, we further investigate the case of exponential family 
distributions and show that under some mild assumptions, the rate
\eqref{eq:rate} holds and is minimax optimal up to
a logarithmic factor.
A matrix completion estimator, defined as
the minimizer of the sum of 
a log-likelihood term and a nuclear norm penalization term,
is first considered.
% It should be noted that the proposed method only
% requires the knowledge of an upper bound of the absolute values
% of the parameter entries.
Provided that
the noise is sub-exponential and the sampling distribution
satisfies some assumptions controlling its deviation from the uniform distribution,
it is proved that with high probability,
the prediction error
is upper bounded by the same rate as in the Gaussian setting \eqref{eq:rate}.
It should be noticed that the sub-exponential assumption is satisfied by
all the above mentioned distributions.

When the additional knowledge of the sampling distribution is available,
we consider 
another estimator, which is inspired by
the one proposed by \cite{Koltchinskii_Tsybakov_Lounici11}
in the additive sub-exponential noise setting.
We adapt their proofs to the exponential family
distributions and show that this estimator satisfies an oracle inequality with respect to
the Kullback-Leibler prediction risk. The proof techniques involved are also closely
related to the dual certificate analysis derived by \cite{Zhang_Zhang12}.
With high probability, an upper bound on the prediction error,
still of the same order as in \eqref{eq:rate},
 is derived from the oracle inequality	. Finally, it is proved that the previous upper bound order
is in fact minimax-optimal up to a logarithmic factor.

The rest of the paper is organized as follows. In \Cref{subsec:model},
the model is specified and some background on exponential family distributions is provided.
Then we give an upper bound
for log\!~-~\!likelihood matrix completion estimator in \Cref{subsec:up} and 
an oracle inequality (also yielding an upper bound)
for the estimator with known sampling scheme in \Cref{subsec:up2}. Finally,
the lower bound is provided in \Cref{subsec:low}. The proofs of the main
results are gathered in \Cref{sec:proof} and the most technical Lemmas
and proofs are deferred to the Appendix.

\subsection*{Notation}
Throughout the paper, the following notation will be used.
For any integers $n,m_1,m_2>0$, $[n] \eqdef \{1,\dots,n\}$,
 $m_1 \vee m_2 \eqdef \max(m_1,m_2)$ and
$m_1 \wedge m_2 \eqdef \min(m_1,m_2)$.
We equip the set of $m_1 \!\times\! m_2$ matrices with real entries
(denoted by $\matset{m_1}{m_2}$) with the Hilbert-Schmidt inner product
$\langle \mat{X}|\mat{X'} \rangle \eqdef \tr(\mat{X}^\top \mat{X'})$.
 For a given matrix $\mat{X}\in \matset{m_1}{m_2}$, we write
$\|\mat{X}\|_\infty \eqdef \max_{i,j} |\mat{X}_{i,j}|$ and for any $\schat_const \geq 1$,
we denote its Schatten $\schat_const$-norm
(see \cite{Bhatia97}) by
 \begin{equation*}
  \|\mat{X}\|_{\sigma,\schat_const}\eqdef 
  \left( \sum_{i=1}^{m_1\wedge m_2} \sigma^\schat_const_i(\mat{X}) \right)^{1/\schat_const} \eqs,
 \end{equation*}
with $\sigma_i(\mat{X})$ the singular values of $\mat{X}$, ordered in decreasing order.
We use the convention $\|\mat{X}\|_{\sigma,\infty}=\sigma_1(X)$.
For any vector $z:=(z_i)_{i=1}^n$, $\diag(z)$ denotes the $\matset{n}{n}$
diagonal matrix whose diagonal entries are $z_1,\cdots,z_n$.
For any convex differentiable function $\Gexp: \RR \to \RR$ and $x,x' \in \RR$,
the Bregman divergence of $\Gexp$ is denoted
by
\begin{equation}
\label{eq:def_bregman_fonc}
 \Breg{G}{x}{x'} \eqdef \Gexp(x) - \Gexp(x') -  \gexp(x')(x-x')\eqs.
\end{equation}

\section{Main results}
\label{sec:main-results}
\subsection{Model Specification}
\label{subsec:model}
We consider an unknown parameter matrix $\mat{\tX} \in \matset{m_1}{m_2}$
that we aim at recovering. Assume that an $\iid$ sequence of indexes $(\omega_i)_{i=1}^n \in ([m_1] \times [m_2])^n$
is sampled and denote by $\Pi$ its distribution.
The observations associated to this sequence are denoted by $(Y_i)_{i=1}^n$
and assumed to follow a natural exponential family distribution, conditionally to the $\mat{\tX}$ entries,
that is:
\begin{equation}
\label{eq:definition-f}
  Y_i|\mat{\tX}_{\omega_i} \sim
  \Exp{h}{\Gexp}(\mat{\tX}_{\omega_i}) \eqdef
  h(Y_i) \exp \left(\mat{\tX}_{\omega_i} Y_i - \Gexp(\mat{\tX}_{\omega_i})\right)\eqsp,
\end{equation}
where $h$ and $\Gexp$ are the base measure and log partition functions
associated to the canonical representation.
For ease of notation we often write $\mat{\tX}_i$  instead of  $\mat{\tX}_{\omega_i}$.

Given two matrices $X^1,X^2\in \matset{m_1}{m_2}$, we define the empirical and 
integrated Bregman divergences as follows
\begin{equation}
\label{eq:def_bregemp}
\Bregemp{\mat{X^1}}{\mat{X^2}}=\frac{1}{n}\sum_{i=1}^n \Breg{G}{ \mat{X^1}_i}{\mat{X^2_i}}
\quad \text{and} \quad \Bregpi{\mat{X^1}}{\mat{X^2}} = \EE[\Bregemp{\mat{X^1}}{\mat{X^2}}] \eqs.
\end{equation}
Note that for exponential family distributions,
the Bregman divergence $\Breg{G}{\cdot}{\cdot}$
corresponds to the Kullback-Leibler divergence. Let
$\PP_{X^1}$ (\resp $\PP_{X^2}$) denote the distribution
of $(Y_1,\omega_1)$ associated to the
parameters $X^1$ (\resp $X^2$); then
$\Bregemp{X^1}{X^2}$ is the Kullback-Leibler divergence
between $\PP_{X^1}$ and $\PP_{X^2}$
conditionally to the sampling, whereas $\Bregpi{X^1}{X^2}$
is the usual Kullback-Leibler divergence.

As reminded in introduction,
the exponential family encompasses a wide range of distributions,
either discrete or continuous. Some information on the most commonly used
is recalled below.\\
\begin{table}[h!]
 \centering
 \begin{tabular}{|l|c|c|}
  \hline
  Distribution & Parameter $x$ & $G(x)$ \\
  \hline
  \textbf{Gaussian:} $\mathcal{N}(\mu,\sigma^2)$ ($\sigma$ known) 
   & $\mu/\sigma$ & $\sigma^2x^2/2$ \\
  \hline
  \textbf{Binomial:} 
  $\mathcal{B}^N(p)$ ($N$ known) & $\log(p/(1-p))$ & $N\log(1+\rme^x)$ \\
  \hline
  \textbf{Poisson:} 
  $\mathcal{P}(\lambda)$ & $\log(\lambda)$ & $\rme^x$ \\
  \hline
  \textbf{Exponential:} 
  $\mathcal{E}(\lambda)$ & $-\lambda$ & $-\log(-x)$ \\
  \hline
\end{tabular}
\caption{Parametrization of some exponential family distributions}
\label{tab:recap}
\end{table}

\begin{remark}
 \label{rem:exp_mom}
If $\Gexp$ is smooth enough, a simple
derivation of the density shows that its
 successive derivatives can be used to determine the distribution moments.
 Thus, when $G$ is twice differentiable,
$\EE[Y_i|\mat{\tX}_{i}] = \gexp(\mat{\tX}_{i})$
and $\mathrm{Var}[Y_i|\mat{\tX}_{i}] =\dgexp(\mat{\tX}_{i})$ hold.
\end{remark}

\subsection{General Matrix Completion}
\label{subsec:up}
In this section, we provide
statistical guarantees on the prediction 
error of a matrix completion estimator,
which is
defined as the minimizer of the sum of a log-likelihood term and
a nuclear norm penalization term. 
For any $X\in \matset{m_1}{m_2}$,
denote by $\Lik(X)$ the (normalized) conditional 
negative log-likelihood of the observations:
\begin{equation}
\label{eq:likelihood-binomial1C}
   \Lik(\mat{X}) = -\frac{1}{n} \sum_{i=1}^{n}\left(
\log(h(Y_i)) + \mat{X}_{i} Y_i - \Gexp(\mat{X}_{i}) \right)\eqsp.
\end{equation}
For $\gamma>0$ and $\lambda>0$, the
nuclear norm penalized estimator $\est$ is defined as follows:
\begin{equation}
\label{eq:MinPb1C}
\est=\argmin_{\substack{\mat{X} \in \RR^{m_1 \times m_2}, \|\mat{X}\|_{\infty}\leq \gamma }}
\Obj(\mat{X}) \eqsp, \quad \text{where} \quad  \Obj(\mat{X})= \Lik(\mat{X}) +
 \lambda  \|\mat{X} \|_{\sigma,1} \eqsp.
\end{equation}
The 
parameter $\lambda$ controls 
the trade off between fitting the data
and privileging a low rank solution: for
large value of $\lambda$, the rank of $\est$
is expected to be small.

Before giving an upper bound on the prediction risk
$\|\est-\tX\|^2_{\sigma,2}$,
the following assumptions
 on the noise and sampling distributions
 need to be introduced.
\begin{assumption}
\label{A0}
The function $x \mapsto \Gexp(x)$, is twice differentiable and strongly convex 
on $[-\gamma, \gamma]$, so that there exists 
constants $\siglo, \sigup >0$ satisfying:
\begin{equation}
 \label{eq:def_K_gamma}
\siglo^2  \leq \dgexp(x) \leq \sigup^2 \eqsp,
\end{equation}
for any $x \in [-\gamma,\gamma]$.
\end{assumption}

\begin{remark}
Under \Cref{A0}, for any $x,x'\in[-\gamma,\gamma]$,
the Bregman divergence
satisfies $\siglo^2(x-x')^2\leq 2 \Breg{\Gexp}{x}{x'} \leq \sigup^2(x-x')^2$.
\end{remark}

\begin{remark}
\label{rem:gaus_strg_cv}
If the observations follow a Gaussian distribution,
the two convexity constants are equal to the standard deviation
\ie $\sigup=\siglo=\sigma$ (see \Cref{tab:recap}).
\end{remark}

For the sampling distribution, one needs to ensure that each entry 
has a sampling probability, which is lower bounded by a strictly positive constant, that is:
\begin{assumption}
\label{A1} There exists a constant $\mu \geq 1$ such that, for all $m_1, m_2$,
\begin{equation}
\label{eq:definition-mu}
\min_{k \in [m_1],\: l \in [m_2]} \pi_{k,l} \geq  1 / (\mu m_1 m_2) \eqsp, 
\quad \text{where  } \pi_{k,l} \eqdef \PP( \omega_1= (k,l))  \eqsp.
\end{equation}
\end{assumption}
Denote by $R_k=\sum_{l=1}^{m_2} \pi_{k,l}$ (\resp $C_{l}=\sum_{k=1}^{m_1} \pi_{k,l}$) 
the probability of sampling a coefficient from row $k$ (\resp column $l$).
The following assumption requires that no line nor column should be
sampled far more frequently than the others.
\begin{assumption}
\label{A2}
There exists a constant $\Lc \geq 1$ such that, for all $m_1, m_2$,
\begin{equation*}
 \max_{k,l}(R_k,C_l) \leq \frac{\Lc}{m_1 \wedge m_2} \eqsp.
\end{equation*}
\end{assumption}
\begin{remark}
 In the classical case of a uniform sampling, $\mu=\Lc=1$ holds.
\end{remark}

We define the sequence of matrices $(E_i)_{i=1}^{n}$,
whose entries are all zeros except for the coefficient $(\omega_i)$
which is equal to one \ie 
$E_i\eqdef~e_{k_i} (e'_{l_i})^\top$ with $(k_i,l_i)=\omega_i$ and
$(e_k)_{k=1}^{m_1}$ (\resp $(e'_l)_{l=1}^{ m_2}$) being the canonical basis of
$\RR^{m_1}$ (\resp\ $\RR^{m_2}$).
Furthermore, for $(\varepsilon_i)_{ i=1}^n$ a Rademacher sequence independent
 from  $(\omega_i,Y_i)_{i=1}^{n}$, we also define
\begin{equation}
\label{eq:def:sigma_rademacher}
 \Sigma_R\eqdef\frac{1}{n}\sum_{i=1}^{n} \varepsilon_i E_i \eqsp,
\end{equation}
and use the following notation
\begin{equation}
\label{eq:definition-d-M}
d= m_1+ m_2 \eqsp, \quad M= m_1 \vee m_2, \quad m= m_1 \wedge m_2  \eqsp.
\end{equation}
With these assumptions and notation, we are now ready
for stating our main results.
\begin{theorem}
\label{th:th_base}
 Assume \Cref{A0}, \Cref{A1}, $\| \mat{\tX} \|_{\infty} \leq \gamma$ and
 $\lambda \geq 2 \| \nabla \Lik(\tX) \|_{\sigma,\infty}$.
 Then with probability at least $1-2d^{-1}$ the following holds:
\begin{equation*}
 \frac{\| \est - \tX \|^2_{\sigma,2}}{ m_1 m_2 }
\leq 
C \mu^2  \max \left( m_1m_2\rank(\tX)\left(\frac{\lambda^2}{\siglo^4} 
+ (\EE \|\Sigma_R \|_{\sigma,\infty})^2 \right) ,
 \frac{ \gamma^2}{\mu} \sqrt{ \frac{\log(d)}{n}}\right) \eqs,
\end{equation*}
 with $\Sigma_R$ and $d$ defined in \eqref{eq:def:sigma_rademacher} and \eqref{eq:definition-d-M} 
and $C$ a numerical constant.
\end{theorem}
\begin{proof}
 See \Cref{proof_base}.
\end{proof}
%C_0C_1=288
%C_0C_2=2048 \rme
In \Cref{th:th_base},
the term $\EE \|\Sigma_R \|_{\sigma,\infty}$
only depends on the sampling
distribution and can be upper bounded
using assumption \Cref{A2}.
On the other hand, the gradient term $ \| \nabla \Lik(\tX) \|_{\sigma,\infty}$ 
depends both on the sampling and on the observation distributions.
In order to control this term with high probability, the noise is assumed to
be sub-exponential.
\begin{assumption}
\label{A3}
There
exist a constant $\sexp_\gamma>0$ such that for all $x\in[-\gamma,\gamma]$ and $Y \sim \Exp{h}{\Gexp}(x)$:
\begin{equation}
\label{eq:def_sub}
\EE\left[\exp\left(\frac{|Y - \gexp(x)|}{\sexp_\gamma}\right) \right] \leq \rme \quad \eqs.
\end{equation}
\end{assumption}
Then \Cref{th:th_base}, \Cref{A2} and \Cref{A3} yield together the following result.
\begin{theorem}
\label{th:th_bis}
 Assume \Cref{A0}, \Cref{A1}, \Cref{A2}, \Cref{A3}, $\| \mat{\tX} \|_{\infty} \leq \gamma$,
 \begin{equation*}
  n \geq 2 \log(d) m \Lc^{-1} \max\left(\frac{\sexp_\gamma^2 }{\sigup^2} 
  \log^2(\sexp_\gamma \sqrt{\frac{m}{\siglo^2}}),
 1/9 \right) \eqs,
 \end{equation*}
 and take $\lambda = 2 c_\gamma \sigup \sqrt{ 2 \Lc \log(d)/(mn)}$,
 where $c_\gamma$ is a constant which depends only on $\sexp_\gamma$.
Then with probability at least $1-3d^{-1}$ the following holds:
 \begin{equation*}
 \frac{\| \est - \tX \|^2_{\sigma,2}}{ m_1 m_2 }
\leq 
\bar{C} \mu^2 \max \left[ \left(\frac{c_\gamma \sigup^2}{\siglo^4}+1\right) \frac{\Lc \rank(\tX)M \log(d)}{n} ,
  \frac{ \gamma^2}{\mu} \sqrt{ \frac{\log(d)}{n}}\right] \eqs,
\end{equation*}
with $\bar{C}$  a numerical constant.
\end{theorem}
\begin{proof}
 See \Cref{proof_bis}.
\end{proof}
\begin{remark}
\label{rem:discus_rate}
When $\gamma$ is treated as a constant
and $n$ is large, the order
 of the bound is
\begin{equation*}
  \frac{\| \est - \tX \|^2_{\sigma,2}}{ m_1 m_2 }
 = \mathcal{O} \left(
%  \mu^2\Lc
% \left(\frac{c_\gamma \sigup^2}{\siglo^4}+1\right) 
\frac{ \rank(\tX)M \log(d)}{n}
\right) \eqs,
\end{equation*}
which matches the rate obtained for Gaussian distributions \eqref{eq:rate}.
Matrix completion
 for exponential family distributions was considered
 in the case of uniform sampling (\ie $\mu=\nu=1$)
  and sub-Gaussian noise by \cite{Gunasekar_Ravikumar_Gosh14}.
 They provide the following upper bound on the estimation error
\begin{equation*}
\frac{\|\tX-\hat{{X}}\|^2_{\sigma,2}}{m_1m_2}
=\mathcal{O} \left( \alpha^{*2} \frac{\rank(\tX) M \log(d) }{n}\right)\eqs.
 % &\precsim \gamma m_1m_2 \frac{\rank(\tX) (m_1 \vee m_2 ) \log(m_1 \vee m_2 ) }{n} \eqsp.
\end{equation*}
with
$\alpha^*$ satisfying $\alpha^* \geq \sqrt{m_1m_2} \|\tX\|_\infty$.
Therefore, \Cref{th:th_bis} improves this rate by a factor $m_1m_2$.
\end{remark}

% \begin{remark}
% In the case of Gaussian noise,
% the constant $\sexp_\gamma=\sigma/\sqrt{2}$ is independent from $\gamma$
% and the strong convexity constants are equal (see \remarkref{rem:gaus_strg_cv}).
% Moreover,
% it can be shown that
% $c_\gamma$ is upper bounded by $6.5$
% (see \cite{Klopp14}).
% \end{remark}

\begin{remark}
\label{rem:unif_vs_subexp}
In the proof, noncommutative Bernstein
 inequality for sub-exponential noise is used to control $ \| \nabla \Lik(\tX) \|_{\sigma,\infty}$.
 However,
 when the observations are uniformly bounded (\eg logistic distribution),
 a uniform Bernstein inequality can be applied instead,
 leading in some cases to a sharper bound (see \cite{Koltchinskii_Tsybakov_Lounici11}
 and \cite{Lafond_Klopp_Moulines_Salmon14} for instance).
\end{remark}

\subsection{Matrix Completion with known sampling scheme}
\label{subsec:up2}
When the sampling distribution $\Pi$ is known,
 the following estimator
 can be defined:
 \begin{align}
  \label{eq:def_estpi}
  &\estpi \eqdef \argmin_{X \in \matset{m_1}{m_2}, \|X\|_{\infty} \leq \gamma} \Likpi(X) + \lambda \|X\|_{\sigma,1} \quad \text{with ,}\\
  &\Likpi(X) \eqdef \Gexp^\Pi(X) - \frac{\sum_{i=1}^n X_i Y_i}{n}   \quad \text{and} \quad
   \Gexp^\Pi(X) \eqdef \EE \left[\frac{\sum_{i=1}^n G(X_i)}{n}\right]\nonumber \eqs.
 \end{align}
% where $A \subset \matset{m_1}{m_2}$ is a nonempty closed bounded convex set.
In the case of sub-exponential additive noise,
\cite{Koltchinskii_Tsybakov_Lounici11} proposed a similar estimator
and have shown that
it satisfies an oracle inequality \wrt the Frobenius prediction risk.
Note that their estimator coincides with \eqref{eq:def_estpi}
for the particular setting of Gaussian noise.
The main interest of computing $\estpi$ instead of
$\est$, when the sampling distribution is known, lies in the fact that a sharp
oracle inequality 
 can be derived for $\estpi$. This powerful tool
allows to provide statistical guarantees 
on the prediction risk,
even if the true parameter $\tX$ does not belong to
the class of estimators \ie when $\|\tX\|\leq \gamma$ is not 
satisfied.
In this section, it is proved that $\estpi$ satisfies
an oracle inequality \wrt the integrated Bregman divergence (see Definition \eqref{eq:def_bregemp}),
which corresponds to the Kullback-Leibler divergence
for exponential family distributions.
An upper
bound on the Frobenius prediction risk is then easily derived from this inequality.
% \begin{remark}
% Note that in the definition \eqref{eq:MinPb1C} of $\est$, the term $\Gexp^\Pi(X)$
% is replaced by its empirical counterpart.
% \end{remark}

% The following assumption
% requires the sampling to be independent
% (not necessarily identically distributed) and
% allows to compare $\Bregpi{\cdot}{\cdot}$
% with $\Bregg{\cdot}{\cdot}$.
% \begin{assumption}
% \label{A4}
%  The coefficients $(w_i)_{i=1}^n$ are sampled independently,
%  and there exists a constant $\bar{\mu}$ such that
%  for all $m_1,m_2$:
%  \begin{equation*}
%   \Bregpi{\cdot}{\cdot} \geq (\bar{\mu}m_1m_2)^{-1} \Bregg{\cdot}{\cdot}
%  \end{equation*}
% \end{assumption}
% Note that \Cref{A4} is weaker than \Cref{A1}. Indeed
% if \Cref{A1} holds for $\mu \geq 1$ then 
% \Cref{A4} holds for $\bar{\mu}=\mu$.
% We can now state the main result of this section
% which proves that $\estpi$ satisfies
% an oracle inequality \wrt the integrated
% Bregman divergence.

\begin{theorem}
 \label{th:oracle_breg}
 Assume \Cref{A0}, \Cref{A1} and $\lambda \geq \|\nabla \Likpi(\tX)\|_{\sigma,\infty}$.
 Then the following inequalities hold:
 \begin{equation}
  \label{eq:oracle_basique}
  \Bregpi{\estpi}{\tX} \leq \inf_{X \in \matset{m_1}{m_2},\|X\|_{\infty} \leq \gamma} \left( \Bregpi{X}{\tX} + 2\lambda \|X\|_{\sigma,1} \right)
 \end{equation}
 and
  \begin{equation}
  \label{eq:oracle_car}
  \Bregpi{\estpi}{\tX} \leq \inf_{X \in \matset{m_1}{m_2},\|X\|_{\infty} \leq \gamma} \left( \Bregpi{X}{\tX} 
   + \left(\frac{1 + \sqrt{2}}{2}\right)^2\frac{\mu}{\siglo^{2}} m_1m_2\lambda^2 \rank(X)\right)
 \end{equation}
\end{theorem}
\begin{proof}
The proof of \Cref{th:oracle_breg}
is an adaptation 
(to exponential family distributions)
of the proof by \cite{Koltchinskii_Tsybakov_Lounici11},
which uses the first order optimality conditions satisfied by
$\estpi$. Similar arguments are used by \cite{Zhang_Zhang12}
to provide dual certificates for non smooth
convex optimization problems.
The detailed proof is given in \Cref{proof:oracle}.
\end{proof}
% We define $\Bregg{\cdot}{\cdot}$ as
% the integrated Bregman divergence when the sampling is uniform, that is
% \begin{equation}
%  \label{eq:def_bregg_unif}
%  \Bregg{X}{X'} \eqdef \frac{\sum_{(k,l) \in [m_1] \times [m_2]} \Breg{G}{X_{k,l}}{X'_{k,l}}}{m_1m_2},
% \end{equation}
% for any $X,X' \in \matset{m_1}{m_2}$.
When $\|\tX\|_{\infty}\leq\gamma$, the previous oracle inequalities imply the following upper bound
on the prediction risk.
\begin{theorem}
\label{th:oracle_up}
 Assume \Cref{A0}, \Cref{A1} and $\lambda \geq \|\nabla \Likpi(\tX)\|_{\sigma,\infty}$
 and $\|\tX\|_{\infty} \leq \gamma$. Then the following holds:
% \begin{equation}
%   \label{eq:or_up_breg2}
%   \frac{\Bregg{\estpi}{\tX}}{m_1m_2} \leq \mu^2 \min \left( \frac{2}{\mu}\lambda \|\tX\|_{\sigma,1},
%    \left(\frac{1 + \sqrt{2}}{2}\right)^2\frac{m_1m_2}{\siglo^{2}} \lambda^2 \rank(\tX)\right)\eqs,
%  \end{equation}
%  and
 \begin{equation}
  \label{eq:or_up_frob2}
  \frac{\|\estpi-\tX\|_{\sigma,2}^2}{m_1m_2} \leq \mu^2 \min \left( 
  \frac{\left(1 + \sqrt{2}\right)^2}{2} \frac{m_1m_2}{\siglo^{4}} \lambda^2 \rank(\tX) ,
  \frac{4}{ \mu\siglo^2} \lambda \|\tX\|_{\sigma,1}
 \right)\eqs.
 \end{equation}
\end{theorem}
\begin{proof}
Applying \Cref{th:oracle_breg} to $X=\tX$ and using \Cref{A1} and \Cref{A0}
yields the result.
\end{proof}

As for the previous estimator,
the term $\|\nabla \Likpi(\tX)\|_{\sigma,\infty}$
is stochastic and depends both on the sampling
and observations.
Assuming that the sampling distribution is uniform and
that the noise is
sub-exponential allows to control it
with high probability. Before stating the result,
let us define
\begin{equation}
 \label{eq:def_lgamma}
 L_\gamma \eqdef \sup_{x \in [-\gamma,\gamma]} |\gexp(x)| \eqs.
\end{equation}
\begin{theorem}
 \label{oracle:prob_up}
 Assume that the sampling is $\iid$ uniform and
 $\|\tX\|_{\infty}\leq\gamma$.
Suppose \Cref{A0}, \Cref{A3},
and
 \begin{equation*}
  n \geq 2 \log(d) m  \max\left(\frac{\sexp_\gamma^2 }{\sigup^2} 
  \log^2(\sexp_\gamma \sqrt{\frac{m}{\siglo^2}}),
 8/9 \right) \eqs.
 \end{equation*}
Take $\lambda =  (c_\gamma \sigup + c^*L_\gamma) \sqrt{ 2 \log(d)/(mn)}$,
 where $c_\gamma$ is a constant which depends only on $\sexp_\gamma$,
 $L_\gamma$ is defined in \eqref{eq:def_lgamma} and $c^*$ is a numerical constant.
Then, with probability at least $1-2d^{-1}$ the following
holds:
% \begin{equation*}
%   \label{eq:or_up_breg3}
%   \frac{\Bregg{\estpi}{\tX}}{m_1m_2} \leq  \min \left( 2\lambda \|\tX\|_{\sigma,1},
%    \left(\frac{1 + \sqrt{2}}{2}\right)^2 \frac{m_1m_2}{ \siglo^{2}}\lambda^2 \rank(\tX)\right)\eqs,
%  \end{equation*}
%  and
 \begin{equation*}
  \label{eq:or_up_frob3}
  \frac{\|\estpi-\tX\|_{\sigma,2}^2}{m_1m_2} \leq \tilde{C} 
%   \min \left(
   \left(\frac{c_\gamma \sigup + L_\gamma}{\siglo^2} \right)^2 \frac{\rank(\tX)M \log(d)}{ n}\lambda^2 \rank(\tX)
%    ,\frac{c_\gamma \sigup + L_\gamma}{\siglo^2}\sqrt{\frac{\rank(\tX) \log(d)}{ mn} } \|\tX\|_{\sigma,1}
%    \right)
   \eqs,
 \end{equation*}
 with $\tilde{C}$ a numerical constant.
\end{theorem}

\begin{remark}
For simplicity we have considered here only the case of uniform
sampling distributions. However if we assume that the sampling satisfies 
\Cref{A1}, \Cref{A2} and that there exists an absolute
constant $\rho$ such that $\pi_{k,l} \leq \rho/\sqrt{m_1m_2}$
for any $m_1,m_2 \in \RR$,
then it is clear from the proof that the same bound still holds
for a general $\iid$ sampling, up
to factors depending on $\mu$, $\Lc$ and $\rho$.
\end{remark}

\begin{remark}
If $\gamma$ is treated as a constant,
the rate obtained 
for the Frobenius error is the same as in \Cref{th:th_bis}.
If not, the two rates might differ because the rate of \Cref{oracle:prob_up}
depends on the constant $L_\gamma$,
which does not appear in \Cref{th:th_bis}.
Note in addition that \remarkref{rem:unif_vs_subexp} also applies to \Cref{oracle:prob_up}.
\end{remark}

\begin{proof}
 The proof is similar
 to the one of \Cref{th:th_bis},
 see \Cref{subsec:proof_th_oracle_up}.
\end{proof}
\subsection{Lower Bound}
\label{subsec:low}
It can be shown that the upper bounds obtained in \Cref{th:th_bis,oracle:prob_up}
are in fact lower bounds (up to a logarithmic factor) when $\gamma$ is treated as a constant.
Before stating the result, let us first introduce the set $\mathcal{F}(r,\gamma)$ of matrices of rank
at most $r$ whose entries are bounded by $\gamma$:
\begin{equation*} %\label{bornes_inf_1}
 \begin{split}
 \mathcal{ F}(r,\gamma)
 &= \left \{\mat{\tX}\in\,\mathbb R^{m_1\times m_2}:\,
 \mathrm{rank}(\mat{\tX})\leq r,\,\Vert\mat{\tX}\Vert_{\infty}\leq \gamma \right \}.
 \end{split}
 \end{equation*}
The infimum over all  estimators
$\mat{\hat{X}}$  that are measurable functions of the data $(\omega_i,Y_i)_{i=1}^{n}$
is denoted by $\inf_{\mat{\hat{X}}}$.

\begin{theorem}
\label{th:th_low}
There exists two constants $c>0$ and $\theta>0$ such that,
 for all $m_1,m_2\geq 2$, $1\leq  r\leq m_1 \wedge m_2$,
and $\gamma>0$,
\begin{equation*}
\inf_{\mat{\hat{X}}}
\sup_{\substack{\mat{\tX}\in\,{\cal F}( r,\gamma)
}}%\hspace{-2mm}
\PP_{\mat{\tX}}\left (\dfrac{\Vert \mat{\hat{X}}-\mat{\tX}\Vert_2^{2}}{m_1m_2} 
> c  \min\left \{\gamma^{2}, \dfrac{Mr}{n\,\sigup^2}\right \} \right )\ \geq \theta \eqsp,
\end{equation*}
\end{theorem}
\begin{remark}
\Cref{th:th_low} provides a lower bound
 of order $\mathcal{O}(Mr/(n\,\sigup^2)$. 
 The order of the ratio between this lower bound and
 the upper bounds of \Cref{th:th_bis} is
$ (c_\gamma (\sigup/\siglo)^4 \log(d)\vee \sigup^2)$.
If $\gamma$ is treated as a constant, 
lower and upper bounds are therefore the same up to a logarithmic
factor.
\end{remark}

\begin{proof}
 See \Cref{proof_low}.
\end{proof}

 \section{Proofs of main results}
\label{sec:proof}
For $\mat{X} \in  \matset{m_1}{m_2}$,  denote by $\mathcal{S}_1(\mat{X}) \subset \RR^{m_1}$ (\resp $\mathcal{S}_2(\mat{X}) \subset \RR^{m_2}$)
the linear spans generated by left (\resp right) singular vectors of $\mat{X}$.
Let $P_{\mathcal{S}^\bot_1(\mat{X})}$  (\resp $P_{\mathcal{S}^\bot_2(\mat{X})}$)
denotes the orthogonal projections on $\mathcal{S}^\bot_1(\mat{X})$ (\resp $\mathcal{S}^\bot_2(\mat{X})$).
We then define the following orthogonal projections on $\RR^{m_1 \times m_2}$
\begin{equation}
\label{eq:def_projec}
\Proj_{\mat{X}}^\bot:\mat{\tilde{X}} \mapsto P_{\mathcal{S}^\bot_1(\mat{X})}\mat{\tilde{X}} P_{\mathcal{S}^\bot_2(\mat{X})}
\text{ and } \Proj_{\mat{X}}: \mat{\tilde{X}} \mapsto \mat{\tilde{X}}-\Proj_{\mat{X}}^\bot(\mat{\tilde{X}})\eqs.
\end{equation}

\subsection{Proof of \Cref{th:th_base} }
\label{proof_base}
From Definition \eqref{eq:MinPb1C},
$\Obj(\est) \leq \Obj(\tX)$ holds, or equivalently
\begin{equation*}
\Bregemp{\est}{\tX} \leq \lambda ( \| \tX \|_{\sigma,1} - \| \est \|_{\sigma,1} )
- \pscal{ \nabla \Lik(\tX)}{\est-\tX} \eqs,
\end{equation*}
with $\Bregemp{\cdot}{\cdot}$ defined in \eqref{eq:def_bregemp}.
The first term of the right hand side can be upper bounded
using \lemmaref{lem:algebre}\speq\eqref{diffschatten} and the second by duality
(between $\|\cdot\|_{\sigma,1}$ and $\|\cdot\|_{\sigma,\infty}$) and 
the assumption on $\lambda$,
which yields
\begin{equation*}
\Bregemp{\est}{\tX} \leq \lambda \left(\| \Proj_{\tX}(\est-\tX) \|_{\sigma,1} 
+ \frac{1}{2} \| \est - \tX \|_{\sigma,1} \right) \eqs.
\end{equation*}
Using \lemmaref{lem:algebre}\speq\eqref{ProjRel} to bound the first term and
\lemmaref{lem:algebre2}\speq\eqref{NucFob} for the second, leads to
\begin{equation}
\label{eq:th_base_up}
 \Bregemp{\est}{\tX} \leq 3 \lambda \sqrt{ 2 \rank( \tX ) } \| \est - \tX \|_{\sigma,2} \eqs. 
\end{equation}
On the other hand, by strong convexity of $G$ (\Cref{A0}), we get
\begin{equation}
\label{eq:def:emp_frob}
 \empf{\est}{\tX} \eqdef \frac{1}{n} \sum_{i=1}^n (\est_i-\tX_i)^2 
 \leq \frac{2}{\siglo^2} \Bregemp{\est}{\tX} \eqs.
\end{equation}
We then define the threshold $\beta \eqdef 8 \rme \gamma^2 \sqrt{\log(d)/n} $ 
and distinguish the two following cases.
\\
\textbf{Case 1} \quad If $\empfw{\est}{\tX} \leq \beta$,
then \lemmaref{lem:delta:frob} yields 
\begin{equation}
\label{eq:proof:th_base:p1}
 \frac{\| \est - \tX \|^2_{\sigma,2} }{m_1m_2} \leq \mu \beta \eqs.
\end{equation}
\textbf{Case 2} \quad If $\empfw{\est}{\tX} > \beta$, then
\lemmaref{lem:algebre2}\speq\eqref{NucFob} and \lemmaref{lem:delta:frob} combined together give\\
$\est \in \rscset{\beta}{32 \mu m_1 m_2  \rank(\tX)}$,
where $\rscset{\cdot}{\cdot}$ is the set defined as
\begin{multline}
\label{eq:def:rscset}
 \rscset{\beta}{r} \eqdef \left\{ \mat{X} \in \matset{m_1}{m_2} \big| \,\, \| \mat{X} - \tX \|_{\sigma,1}
\leq \sqrt{ r \EE \left[ \empf{\mat{X}}{\tX} \right] } ;
\EE \left[ \empf{\mat{X}}{\tX} \right] > \beta
\right\} \eqs.
 \end{multline}
 Hence, from \lemmaref{DevUnifCont}
 it holds, with probability at least $1-(d-1)^{-1} \geq 1-2d^{-1}$, that

\begin{equation}
\label{eq:proof:th_base:p2a}
\empf{\mat{X}}{\tX} \geq \frac{1}{2} \EE\left[ \empf{\mat{X}}{\tX} \right] - 
 512 \rme (\EE \|\Sigma_R \|_{\sigma,\infty})^2\mu m_1 m_2 \rank(\tX) \eqs.
\end{equation}
Combining \eqref{eq:proof:th_base:p2a} with \eqref{eq:def:emp_frob}, \eqref{eq:th_base_up}
and \lemmaref{lem:delta:frob}
leads to
\begin{equation}
\label{eq:proof_th_base_almostlast}
 \frac{\| \est - \tX \|^2_{\sigma,2}}{2\mu m_1 m_2 } -512 \rme (\EE \|\Sigma_R \|_{\sigma,\infty})^2\mu m_1 m_2 \rank(\tX)
 \leq \frac{6 \lambda}{\siglo^2} \sqrt{ 2m_1 m_2 \rank( \tX ) } \frac{\| \est - \tX \|_{\sigma,2}}{\sqrt{m_1 m_2}} \eqs.
\end{equation}
% A sign analysis of this second-order polynomial yields
% \begin{equation}
% \label{eq:proof:th_base:p2b}
% \frac{\| \est - \tX \|_{\sigma,2}}{ \sqrt{m_1 m_2} }
% \leq \mu \sqrt{m_1 m_2} \left( \frac{6 \lambda \sqrt{2 \rank(\tX)}}{\siglo^2}
% + \sqrt{ \frac{72 \lambda^2 \rank(\tX)}{\siglo^4} +  \epsilon(64 \rank(\tX))}\right) \eqs.
% \end{equation}
Using the identity $ab \leq a^2 + b^2/4$  in \eqref{eq:proof_th_base_almostlast} and combining with
\eqref{eq:proof:th_base:p1} achieves the proof of \Cref{th:th_base}.

% ################### *************  LEMMAS ALGEBRA    *****************##########################``

\begin{lemma}
\label{lem:algebre}
 For any pair of matrices $\mat{X}, \, \mat{\tilde{X}} \in \matset{m_1}{m_2}$ we have
 \begin{enumerate}[(i)]
  \item $ \|\mat{X} + \Proj_{\mat{X}}^\bot(\mat{\tilde{X}}) \|_{\sigma,1}=\|\mat{X}\|_{\sigma,1} 
  +  \|\Proj_{\mat{X}}^\bot(\mat{\tilde{X}}) \|_{\sigma,1}\eqs,$ \label{PenRel}
  \item $\| \Proj_{\mat{X}}(\mat{\tilde{X}}) \|_{\sigma,1} 
  \leq \sqrt{2 \rank(\mat{X})} \|\mat{\tilde{X}}\|_{\sigma,2}\eqs,$ \label{ProjRel}
  \item  $ \|\mat{X}\|_{\sigma,1}-\|\mat{\tilde{X}}\|_{\sigma,1}
  \leq  \|\Proj_{\mat{X}}(\mat{\tilde{X}}-\mat{X})\|_{\sigma,1}\eqs.$ \label{diffschatten}
  \end{enumerate}

\end{lemma}

\begin{lemma}
\label{lem:algebre2}
Let $\mat{X},\mat{\tilde{X}} \in \matset{m_1}{m_2}$ satisfying 
$\|\mat{X}\|_{\infty}\leq \gamma$ and $\|\mat{\tilde{X}}\|_{\infty}\leq \gamma$.
Assume that $\lambda>2 \| \nabla \Lik( \tX ) \|_{\sigma,\infty}$
and $\Obj(\mat{X}) \leq \Obj(\mat{\tilde{X}})$. Then
\begin{enumerate}[(i)]
 \item $\|\Proj_{\mat{\tilde{X}}}^\bot(\mat{X}-\mat{\tilde{X}})\|_{\sigma,1} 
\leq 3 \|\Proj_{\mat{\tilde{X}}}(\mat{X}-\mat{\tilde{X}})\|_{\sigma,1}\eqs,$  \label{projrel} 
 \item $\|\mat{X}-\mat{\tilde{X}}\|_{\sigma,1}
\leq 4 \sqrt{2 \rank(\mat{\tilde{X}})} \|(\mat{X}-\mat{\tilde{X}})\|_{\sigma,2}\eqs.$ \label{NucFob} 
\end{enumerate}
\end{lemma}
\begin{lemma}
\label{lem:delta:frob}
 Under \Cref{A1}, for any $X \in \matset{m_1}{m_2}$ it holds
 \begin{equation*}
\empfw{X}{\tX}
\geq  \frac{1}{\mu m_1 m_2} \| X - \tX \|_{\sigma,2}^2 \eqs.
\end{equation*}
\end{lemma}
% \begin{proof}
% 
% % \begin{equation*}
% % \EE \left[ \empf{\est}{\tX} \right] =\frac{1}{n} \sum_{i=1}^{n}\sum_{\substack{k \in [m_1] \\ l \in [m_2]}}
% %  \pi_{k,l} (\est_{k,l} - \tX_{k,l})^2
% %  \geq  \frac{1}{\mu m_1 m_2} \| \est - \tX \|_{\sigma,2}^2 \eqs,
% % \end{equation*}
% % where \Cref{A1} was used for the last inequality.
% 
% % We now recall the definitions of the following sets
% % \begin{align*}
% %  \mathcal{C}(r)\eqdef&\left\{\mat{X} \in  \RR^{m_1 \times m_2} : \:\: \|\mat{X}\|_{\infty}\leq \gamma, \: \|\mat{X}-\mat{\tX}\|_{\sigma,1}^2\leq r \D\left(f(\mat{\tX}),f(\mat{X})\right) \right\}\eqs, \\
% %   \mathcal{C}_\beta(r)\eqdef&\left\{ \mat{X} \in  \RR^{m_1 \times m_2} : \:\: \mat{X} \in \mathcal{C}(r), \: \D\left(f(\mat{\tX}),f(\mat{X})\right) > \beta \right\}.
% % \end{align*}
% \end{proof}
% ################### *************  LEMMAS DEVIATION    *****************##########################``
\begin{lemma}
 \label{DevUnifCont}
For $\beta= 8 \rme \gamma^2 \sqrt{\log(d)/n}$,
 with probability at least $1- (d-1)^{-1}$,   
we have for all $\mat{X} \in \rscset{\beta}{r}$:
 \begin{equation*}
   \left|\empf{X}{\tX} - \EE\left[\empf{X}{\tX}\right] \right|
   \leq \frac{\EE\left[\empf{X}{\tX}\right]}{2}+16 \rme (\EE \|\Sigma_R \|_{\sigma,\infty})^2\eqs r \eqs,
 \end{equation*}
 with $\rscset{\beta}{r}$ defined in \eqref{eq:def:rscset}.
\end{lemma}
\begin{proof}
 \lemmaref{lem:algebre,lem:algebre2} are proved in \Cref{ap:lemma_algebre}.
 \lemmaref{lem:delta:frob} follows directly from \Cref{A1}.
  See \Cref{ap:cont_sto} for the proof of \lemmaref{DevUnifCont}.
\end{proof}

 \subsection{Proof of \Cref{th:th_bis} }
 \label{proof_bis}
 Starting from \Cref{th:th_base} one only needs to control
 $\EE(\|\Sigma_R\|_{\sigma,\infty})$ and
 $\| \nabla \Lik (\tX) \|_{\sigma,\infty}$ to obtain the result.\\
 \textbf{Control of $\EE(\|\Sigma_R\|_{\sigma,\infty})$:}\quad
One can write  $\Sigma_R\eqdef n^{-1}\sum_{i=1}^{n} Z_i$, with $Z_i\eqdef \varepsilon_i E_i$
  which satisfies $\EE[Z_i]=0$. 
  Recalling the definitions
 $R_k\!=\sum_{l=1}^{m_2} \pi_{k,l}$ and $C_{l}\!=\sum_{k=1}^{m_1} \pi_{k,l}$ for any  $k\in [m_1]$,
$l \in [m_2]$, one obtains
\begin{equation}
\label{control:ezi}
\norm{\EE\left[\frac{1}{n} \sum_{i=1}^{n}  \mat{Z_{i}}\mat{Z_{i}}^\top \right]}_{\sigma,\infty}
\leq   \norm{\diag((R_k)_{k=1}^{m_1})}_{\sigma,\infty} \leq   \frac{\Lc}{m} \eqs,
\end{equation}
where H\ref{A2} was used for the last inequality.
Using a similar argument one also gets
$\|\EE[\sum_{i=1}^{n} \mat{Z_{i}}^\top\mat{Z_{i}} ]\|_{\sigma,\infty}/n\leq  \Lc/m$.
Hence applying \lemmaref{lem:MatExp}
with $U=1$ and $\sigma_Z^2 =\nu/m$, for $n\geq m\log(d)/(9 \Lc)$ yields
\begin{equation}
 \label{proof:th2ineq2}
  \EE \left[\|\Sigma_R  \|_{\sigma,\infty} \right] \leq c^{*}\sqrt{\frac{2\rme\Lc \log(d)}{mn}} \eqs,
\end{equation}
with $c^*$ a numerical constant.\\
 \textbf{Control of  $\| \nabla \Lik (\tX) \|_{\sigma,\infty}$:}\quad
 Let us define $Z'_i\eqdef (Y_i - \gexp(\tX_i)) E_i$,
 which satisfies
 $\nabla \Lik (\tX)\eqdef n^{-1}\sum_{i=1}^{n} Z'_i$ and $\EE[Z'_i]=0$ (as any score function)
 and
 \begin{equation*}
 \label{eq:def_sigma_z_prime}
  \sigma^2_{Z'} \eqdef \max \left(
  \frac{1}{n}\|\EE[\sum_{i=1}^{n} (\mat{Z'_{i}})^\top\mat{Z'_{i}} ]\|_{\sigma,\infty} \:,\:
  \frac{1}{n}\|\EE[\sum_{i=1}^{n} \mat{Z'_{i}}(\mat{Z'_{i}})^\top ]\|_{\sigma,\infty} 
  \right) \eqs.
 \end{equation*}
Using H\ref{A3}, a similar analysis yields $\sigma^2_{Z'} \leq  \sigup^2\Lc/m$.
On the other hand, $\max_{k,l}(R_k,C_l) \geq 1/m$
and $\EE[(Y_i - \gexp(\tX_i))^2] = \dgexp(\tX_i) \geq \siglo^2$ gives
$\sigma^2_{Z'} \geq \siglo^2 /m$.
Applying \propositionref{prop:bernstein_exp} for $t=\log(d)$ gives with probability at least 
$1-d^{-1}$
\begin{equation}
\label{control:score}
\| \nabla \Lik (\tX) \|_{\sigma,\infty} \leq 
c_\gamma \max \left\{ \sigup\sqrt{\Lc/m} \sqrt{ \frac{2 \log(d)}{n}} ,  
  \sexp_\gamma \log(\frac{\sexp_\gamma \sqrt{m}}{\siglo})\frac{2\log(d)}{n} \right \},
\end{equation}
with $c_\gamma$ which depends only on $\sexp_\gamma$.
By assumption on $n$, the left term dominates. Therefore
taking $\lambda$ as in \Cref{th:th_bis} statement
yields $\lambda \geq 2 \| \nabla \Lik (\tX) \|_{\sigma,\infty}$ with probability
at least $1-d^{-1}$. A union bound argument combined to 
\Cref{th:th_base} achieves \Cref{th:th_bis} proof.

\begin{lemma}
 \label{lem:MatExp}
 Consider a finite sequence of independent random matrices 
 $(\Bern{Z_{i}})_{1 \leq i \leq n}\in \RR^{m_1 \times m_2}$ satisfying $\EE[\Bern{Z_{i}}]=0$ 
 and for some $U>0$,
 $\| \Bern{Z_{i}} \|_{\sigma,\infty} \leq U$ for all $i= 1, \dots, n$ and define
\begin{equation*}
 \sigma^2_Z \eqdef \max \left\{\left\| \frac{1}{n} \sum_{i=1}^{n} 
 \EE[  \Bern{Z_{i}}\Bern{Z_{i}^\top} ] \right\|_{\sigma,\infty},
 \left\| \frac{1}{n}\sum_{i=1}^{n} \EE[ \Bern{Z_{i}}^\top\Bern{Z_{i}} ]\right\|_{\sigma,\infty}\right\} \eqs.
\end{equation*}
Then, for any $n \geq (U^2\log(d))/(9\sigma^2_Z)$ the following holds:
\begin{equation*}
  \EE \left[\norm{\frac{1}{n} \sum_{i=1}^{n} \Bern{Z_{i}}}_{\sigma,\infty} \right]
  \leq c^{*}\sigma_Z\sqrt{\frac{2 \rme \log(d)}{n}} \eqs,
\end{equation*}
with $c^*=1+\sqrt{3}$.
\end{lemma}

\begin{proof}
See \cite{Klopp_Lafond_Moulines_Salmon14}[Lemma 15].
\end{proof}

%%%%%%%%%%%%%%%%%%%%%%%%%%%  SUBEXP BERNSTEIN %%%%%%%%%%%%%%%%%%%%%%%%%%%%%%
 \begin{proposition}
\label{prop:bernstein_exp}
 Consider a finite sequence of independent 
 random matrices $(\Bern{Z_{i}})_{1 \leq i \leq n}\in \RR^{m_1 \times m_2}$ 
 satisfying $\EE[\Bern{Z_{i}}]=0$. For some $U>0$,
 assume
 \begin{equation*}
   \inf \{ \sexp >0: \EE [ \exp (\| \Bern{Z_{i}} \|_{\sigma,\infty}/ \sexp) ] \leq \rme \}
   \leq U \quad \text{for} \quad i= 1, \dots, n 
 \end{equation*}
 and define $\sigma_Z$ as in \lemmaref{lem:MatExp}.
Then for any $t>0$,
 with probability at least $1-\rme^{-t}$
 \begin{equation*}
  \norm{\frac{1}{n} \sum_{i=1}^{n} \Bern{Z_{i}}}_{\sigma,\infty} 
  \leq c_U \max \left\{ \sigma_Z \sqrt{ \frac{t + \log(d)}{n}} ,  
  U \log(\frac{U}{\sigma_Z})\frac{t + \log(d)}{n} \right \} \eqs,
 \end{equation*}
with $c_U$ a constant which depends only on $U$.

\end{proposition}
\begin{proof}
This result is an extension of the sub-exponential
noncommutative Bernstein inequality \cite[Theorem 4]{koltchinskii13},
to rectangular matrices by dilation, see \cite[Proposition~11]{Klopp14}
for details.
\end{proof}
\subsection{Proof of \Cref{th:th_low}}
\label{proof_low}
We start with a packing set construction, inspired by \cite{Koltchinskii_Tsybakov_Lounici11}.
Assume \wlg that $m_1\geq m_2$.
Let $\alpha \in (0,1/8)$ 
and define $\kappa \eqdef \min(1/2, \sqrt{\alpha m_1 r}/(2 \gamma \sigup^2 \sqrt{n})$
and the set of matrices
\begin{equation*}
\mathcal{L} \, =\left \{ L = (l_{ij})\in\RR^{m_1\times r}:
l_{ij}\in\left \{0, \kappa\gamma\right \}\,, \forall i \in [m_1],\, \forall j \in [r] \right \}.
\end{equation*}
Consider the associated set of block matrices
$$
\mathcal{L}' \ =\ \Big\{
L'=(\begin{array}{c|c|c|c}L&\cdots&L&O
\end{array})\in\matset{m_1}{m_2}: L\in \mathcal{L}\Big\},
$$
where $O$ denotes the $m_1\times (m_2-r\lfloor m_2/r \rfloor )$ zero
matrix, and $\lfloor x \rfloor$ is the integer part of $x$.
% \begin{remark}
% In the case $m_1< m_2$, we only need to change the construction of the low rank component of the test set.
% We first build a matrix $\tilde L  \in \mathbb R^{r \times m_2} $  with
% entries in $\left \{0, \kappa\gamma\right \}$ and
% then we replicate this matrix to obtain a block matrix $L$ of size $m_1 \times m_2$.
% \end{remark}
 The Varshamov-Gilbert bound (\cite[Lemma 2.9]{Tsybakov09}) guarantees the existence 
 of a subset $\mathcal{A}\subset\mathcal{L}'$ with
cardinality $\mathrm{Card}(\mathcal{A}) \geq 2^{(rm_1)/8}+1$ 
containing the null matrix $X^0$ and such that, for any two
distinct elements $\mat{X^1}$ and $\mat{X^2}$ of $\mathcal{A}$,
\begin{equation} \label{lower_2}
\Arrowvert \mat{X^1}-\mat{X^2}\Arrowvert_{2}^2  \geq \frac{m_1r\,\kappa^2\gamma^{2}}{8}
\left\lfloor \frac{m_2}{r}\right\rfloor \geq
\frac{m_1m_2\,\kappa^2\gamma^{2}}{16}\,.
\end{equation}
By construction, any element of $\mathcal{A}$ as well
as the difference of any two elements of $\mathcal{A}$ has rank at most $r$, the entries of any matrix in
$\mathcal{A}$ take values in $[0,\gamma]$ and thus
$\mathcal{A}\subset {\mathcal F}(r,\gamma)$.
For some $\mat{X} \in \AA$, we now estimate  the Kullback-Leibler
divergence $\KLd{\PP_{\mat{X}}\!}{\PP_{\mat{X^0}}\!}$ between probability measures $\PP_{\mat{X^0}}$
and $\PP_{\mat{X}}$. By independence of the observations $(Y_i,\omega_i)_{i=1}^n$ and 
since the distribution of $Y_i|\omega_i$ belongs to the exponential family one obtains
\begin{equation*}
 \KLd{\PP_{\mat{X}}}{\PP_{\mat{X}^0}}
 =n\EE_{\omega_1} \left[ G'(X_{\omega_1})(X_{\omega_1}-X^0_{\omega_1})
 -G(X_{\omega_1})+G(X^0_{\omega_1})\right] \eqsp.
\end{equation*}
Since $\mat{X^0_{\omega_1}}= 0$ and
either $\mat{X_{\omega_1}}=0$ or $\mat{X_{\omega_1}}=\kappa\gamma$,
by strong convexity and by definition of $\kappa$ one gets
\begin{equation*} %\label{KLdiv}
\KLd{\PP_{\mat{X}}}{\PP_{\mat{X}^0}}
\leq n \frac{\sigup^2}{2} \kappa^2 \gamma^2 
\leq \frac{\alpha r m_1}{8} \leq \alpha \log_2(\mathrm{Card}(\mathcal{A})-1) \eqs,
\end{equation*}
which implies
\begin{equation}\label{eq: condition C}
\frac{1}{\mathrm{Card}(\AA)-1} \sum_{\mat{X} \in \AA}\KLd{\PP_{\mat{X^0}}}{\PP_{\mat{X}}}\ \leq\ \alpha \log \big(\mathrm{Card}(\AA)-1\big) \eqsp.
\end{equation}
Using \eqref{lower_2}, \eqref{eq: condition C} and \cite[Theorem 2.5]{Tsybakov09} together gives
\begin{equation*}
\inf_{\mat{\hat{X}}}
\sup_{\substack{\mat{\tX}\in\,{\cal F}( r,\gamma)
}}%\hspace{-2mm}
\PP_{\mat{\tX}}\left (\dfrac{\Vert \mat{\hat{X}}-\mat{\tX}\Vert_2^{2}}{m_1m_2} 
> \tilde{c}  \min\left \{\gamma^{2}, \dfrac{\alpha M r}{n\,\sigup^2}\right \} \right )\ \geq\ \delta(\alpha,M) \eqsp,
\end{equation*}
where
\begin{equation}
\label{eq;definition-delta}
\delta(\alpha,M)=\frac{1}{1+2^{-rM/16}}\left(1-2\alpha-\frac{1}{2} \sqrt{\frac{\alpha}{rM\log(2)}} \right) \eqsp,
\end{equation}
and $\tilde{c}$ is a numerical constant.
Since we are free to choose $\alpha$ as small as possible, this achieves the proof.
% \begin{equation}\label{lower_1}
% \inf_{\mat{\hat{X}}}
% \sup_{\substack{\mat{\tX}\in\,{\cal F}( r,\gamma)
% }}%\hspace{-2mm}
% \mathbb P_{\mat{\tX}}
% \left (\dfrac{\Vert \mat{\hat{X}}-\mat{\tX}\Vert_2^{2}}{m_1m_2}
% > c\min\left \{\gamma^{2}, \dfrac{m_1r}{n\,\sigup^2}\right \} \right )\ \geq\ \delta(\alpha,m_1)
% \end{equation}
% for some  universal constants $c>0$ and $\delta\in(0,1)$.

\bibliography{references_all}
  \appendix
\section{Proof of \lemmaref{lem:algebre} and \lemmaref{lem:algebre2}}
\label{ap:lemma_algebre}
\textbf{\lemmaref{lem:algebre}}\\
 \begin{proof}
If $A, B \in \rset^{m_1 \times m_2}$ are two matrices satisfying $\mathcal{S}_i(A) \perp \mathcal{S}_i(B)$, $i=1,2$,
(see Definition \eqref{eq:def_projec})
then $\| A + B \|_{\sigma,1}= \| A \|_{\sigma,1} + \| B \|_{\sigma,1}$. Applying this identity with $A = X$ and $B= \Proj_{\mat{X}}^\bot(\mat{\tilde{X}})$, we obtain
 \[
 \|X + \Proj_{\mat{X}}^\bot(\mat{\tilde{X}}) \|_{\sigma,1}
 = \|X\|_{\sigma,1} +  \|\Proj_{\mat{X}}^\bot(\mat{\tilde{X}}) \|_{\sigma,1}\eqs,
 \]
 showing~\eqref{PenRel}.

 From the definition of $\Proj_{\mat{X}}(\cdot)$,
 $\Proj_{\mat{X}}(\mat{\tilde{X}})=
 P_{\mathcal{S}_1(\mat{X})}\mat{\tilde{X}}
 P_{\mathcal{S}^\bot_2(\mat{X})}+\mat{\tilde{X}}P_{\mathcal{S}_2(\mat{X})}$ holds
 and therefore $\rank(\Proj_{\mat{X}}(\tilde{X})) \leq 2 \rank(X)$.
 On the other hand, 
 the Cauchy-Schwarz inequality implies that for any matrix $A$,  
 $\|A\|_{\sigma,1} \leq \sqrt{ \rank(A)} \|C\|_{\sigma,2}$. 
 Consequently  \eqref{ProjRel} follows from
 \begin{align*}
 \| \Proj_{\mat{X}}(\mat{\tilde{X}}) \|_{\sigma,1} &\leq \sqrt{2 \rank(X)} \|\Proj_{\mat{X}}(\mat{\tilde{X}})\|_{\sigma,2}\eqs \leq \sqrt{2 \rank(X)} \|\mat{\tilde{X}}\|_{\sigma,2}\eqs \eqsp.
 \end{align*}
 Finally, since $\mat{\tilde{X}}=\mat{X} + \Proj_{\mat{X}}^\bot(\mat{\tilde{X}}-\mat{X})+\Proj_{\mat{X}}(\mat{\tilde{X}}-\mat{X})$ we have
 \begin{align*}
  \|\mat{\tilde{X}} \|_{\sigma,1} &\geq \|\mat{X} + \Proj_{\mat{X}}^\bot(\mat{\tilde{X}}-\mat{X}) \|_{\sigma,1} -\|\Proj_{\mat{X}}(\mat{\tilde{X}}-\mat{X}) \|_{\sigma,1} \eqs, \\
  &= \|\mat{X} \|_{\sigma,1} + \| \Proj_{\mat{X}}^\bot(\mat{\tilde{X}}-\mat{X}) \|_{\sigma,1} -\|\Proj_{\mat{X}}(\mat{\tilde{X}}-\mat{X}) \|_{\sigma,1} \eqs,
 \end{align*}
 leading to \eqref{diffschatten}.
 \end{proof}
 \textbf{\lemmaref{lem:algebre2}}\\
 \begin{proof}
Since $\Obj(\mat{X}) \leq \Obj(\mat{\tilde{X}})$, we have
 \begin{equation*}
 \Lik(\mat{\tilde{X}}) -\Lik(\mat{X})  \geq \lambda 
 (\|\mat{X} \|_{\sigma,1}-\|\mat{\tilde{X}} \|_{\sigma,1}).
 \end{equation*}
For any $\mat{X} \in \matset{m_1}{ m_2}$, 
using $\mat{X}=\mat{\tilde{X}}+\Proj_{\mat{\tilde{X}}}^\bot(\mat{X}-\mat{\tilde{X}})
+\Proj_{\mat{\tilde{X}}}(\mat{X}-\mat{\tilde{X}})$, \lemmaref{lem:algebre}\speq\eqref{PenRel}
and the triangular inequality, we get
\begin{equation*}
 \|\mat{X}\|_{\sigma,1} \geq \|\mat{\tilde{X}}\|_{\sigma,1} 
 + \|\Proj_{\mat{\tilde{X}}}^\bot(\mat{X}-\mat{\tilde{X}})\|_{\sigma,1}
 -\|\Proj_{\mat{\tilde{X}}}(\mat{X}-\mat{\tilde{X}})\|_{\sigma,1} \eqs,
\end{equation*}
which implies
\begin{equation}
 \label{LemProj:lb}
 \Lik(\mat{\tilde{X}}) -\Lik(\mat{X})  \geq  \lambda
   \left( \|\Proj_{\mat{\tilde{X}}}^\bot(\mat{X}-\mat{\tilde{X}})\|_{\sigma,1}
  -\|\Proj_{\mat{\tilde{X}}}(\mat{X}-\mat{\tilde{X}})\|_{\sigma,1} \right)\eqs.
\end{equation}
Furthermore by convexity of $\Lik$ we have
\begin{align*}
 \Lik(\mat{\tilde{X}}) -\Lik(\mat{X})  \leq  \pscal{\nabla \Lik( \mat{\tilde{X}}) }
 {\mat{\tilde{X}}-\mat{X}} \eqs,
\end{align*}
% The duality between $\|\cdot\|_{\sigma,1}$ and $\|\cdot\|_{\sigma,\infty}$ 
% % (see for instance
% % \cite[Corollary IV.2.6]{Bhatia97})
% leads to
which yields by duality
\begin{align}
 \Lik(\mat{\tilde{X}}) -\Lik(\mat{X}) &\leq \| \nabla \Lik( \mat{\tilde{X}}) \|_{\sigma,\infty}
  \|\mat{\tilde{X}} -\mat{X} \|_{\sigma,1}
 \leq \frac{\lambda}{2}  \|\mat{\tilde{X}}-\mat{X} \|_{\sigma,1} \eqs, \nonumber \\
& \leq \frac{\lambda}{2}
(\|\Proj_{\mat{\tilde{X}}}^\bot(\mat{X}-\mat{\tilde{X}})\|_{\sigma,1}
+ \|\Proj_{\mat{\tilde{X}}}(\mat{X}-\mat{\tilde{X}})\|_{\sigma,1})\eqs, \label{LemProj:ub}
\end{align}
where we used $\lambda> \| \nabla \Lik( \mat{\tilde{X}}) \|_{\sigma,\infty}$
in the second line. Then combining \eqref{LemProj:lb} with \eqref{LemProj:ub} gives \eqref{projrel}.
Since 
$\mat{X}-\mat{\tilde{X}}=\Proj_{\mat{\tilde{X}}}^\bot(\mat{X}-\mat{\tilde{X}})
+\Proj_{\mat{\tilde{X}}}(\mat{X}-\mat{\tilde{X}})$,
using the triangular inequality and \eqref{projrel} yields
\begin{equation}
\label{LemProj:int}
 \|\mat{X}-\mat{\tilde{X}}|_{\sigma,1} \leq
 4 \|\Proj_{\mat{\tilde{X}}}(\mat{X}-\mat{\tilde{X}})\|_{\sigma,1}.
\end{equation}
Combining \eqref{LemProj:int} and  \lemmaref{lem:algebre}\speq\eqref{projrel} leads to \eqref{NucFob}.
\end{proof}

\section{Proof of \lemmaref{DevUnifCont}}
\label{ap:cont_sto}
\begin{proof}
The proof is adapted from \cite[Theorem~1]{Negahban_Wainwright12} and \cite[Lemma~12]{Klopp14}.
We use a peeling argument combined with a sharp deviation inequality detailed in \Cref{SliceCont}.
For any $\alpha>1$, $\beta>0$ and $0<\eta<1/2\alpha$,
define 
\begin{equation}\label{eq:def_eps}
 \epsilon(r,\alpha,\eta) \eqdef\frac{4 }{1/(2\alpha)-\eta}(\EE \|\Sigma_R \|_{\sigma,\infty})^2r\eqs,
\end{equation}
and consider the events
\begin{multline*}
 \mathcal{B} \eqdef \bigg\{ \exists \mat{X} \in \rscset{\beta}{r}\bigg |
  \left|\empf{X}{\tX}- \EE\left[\empf{X}{\tX}\right] \right|
 > \frac{\EE\left[\empf{X}{\tX}\right]}{2}+\epsilon(r,\alpha,\eta) \bigg\}\eqs,
\end{multline*}
and
\begin{equation*}
 \mathcal{R}_l\eqdef \left\{ \mat{X} \in \rscset{\beta}{r}| \alpha^{l-1}\beta 
 < \EE\left[\empf{X}{\tX}\right]
 < \alpha^{l}\beta \right\}\eqs.
\end{equation*}
Let us also define the set
\begin{equation*}
  \rscsett{\beta}{r}{t} \eqdef \left\{ \mat{X} \in  \rscset{\beta}{r}|\:\: 
  \EE\left[\empf{X}{\tX}\right] \leq t \right\} \eqs,
\end{equation*}
and
\begin{equation}\label{eq:def_Zt}
 Z_t \eqdef \sup_{\mat{X} \in  \rscsett{\beta}{r}{t} }
 |\empf{X}{\tX}- \EE\left[\empf{X}{\tX}\right] | \eqs.
\end{equation}
Then for any $\mat{X} \in  \mathcal{B} \cap \mathcal{R}_l$ we have
\begin{equation*}
 |\empf{X}{\tX}- \EE\left[\empf{X}{\tX}\right]  | 
 > \frac{1}{2}\alpha^{l-1}\beta + \epsilon(r,\alpha,\eta)\eqs,
\end{equation*}
Moreover by definition of $\mathcal{R}_l$, $\mat{X} \in \mathcal{C}_\beta (r,\alpha^{l}\beta)$.
Therefore

\begin{equation*}
 \mathcal{B} \cap \mathcal{R}_l \subset \mathcal{B}_l \eqdef \{ Z_{\alpha^{l}\beta} > \frac{1}{2\alpha}\alpha^{l}\beta+\epsilon(r,\alpha,\eta) \}\eqs,
\end{equation*}
If we now apply a union bound argument combined to \lemmaref{SliceCont} we get
 \begin{equation*}
 \PP(\mathcal{B}) \leq \sum_{l=1}^{+\infty} \PP(\mathcal{B}_l)\leq 
 \sum_{l=1}^{+\infty} \exp\left(-\frac{n\eta^2(\alpha^{l}\beta)^2}{8\gamma^4}\right)
 \leq \frac{\exp(-\frac{n\eta^2\log(\alpha)\beta^2}{4\gamma^4})}{1-\exp(-\frac{n\eta^2\log(\alpha)\beta^2}{4\gamma^4})}\eqs,
 \end{equation*}
 where we used $x\leq \rme^x$ in the second inequality. Choosing
 $\alpha = \rme$, $\eta= (4 \rme)^{-1}$ 
and $\beta$ as stated in the Lemma yields the result.
\end{proof}

\begin{lemma}
 \label{SliceCont}
Let $\alpha>1$ and $0<\eta<\frac{1}{2\alpha}$. Then we have
\begin{equation}
\label{eq:SliceCont}
 \PP\left(Z_t > t/(2\alpha) + \epsilon(r,\alpha,\eta) \right) \leq \exp\left(-n\eta^2t^2/(8\gamma^4)\right)\eqs,
\end{equation}
where $\epsilon(r,\alpha,\eta)$ and $Z_t$
are defined in \eqref{eq:def_eps} and \eqref{eq:def_Zt}.
% \begin{equation*}
%  \epsilon(r,\alpha,\eta) \eqdef\frac{4 L_\gamma^2r}{1/(2\alpha)-\eta}(\EE \|\Sigma_R \|_{\sigma,\infty})^2\eqs.
% \end{equation*}
 \end{lemma}

 \begin{proof}
From Massart's inequality (\cite[Theorem 9]{Massart00})
we get for  $0<\eta<1/(2\alpha)$
\begin{equation}
\label{proof:MassartConc}
 \PP(Z_t> \EE[Z_t]+ \eta t)\leq \exp\left(-\eta^2nt^2/(8\gamma^4) \right)\eqs.
\end{equation}
A symmetrization argument gives
\begin{equation*}
 \EE[Z_t]\leq 2 \EE\left[\sup_{\mat{X} \in \rscsett{\beta}{r}{t}} \left|\frac{1}{n}\sum_{i=1}^{n}
 \varepsilon_i ( X_i - \tX_i)^2 \right|\right]\eqs,
\end{equation*}
where $\bvarepsilon \eqdef (\varepsilon_i)_{1 \leq i\leq n}$ is a Rademacher 
sequence independent from $(Y_i,\omega_i)_{i=1}^n$. 
The contraction principle  (\cite[Theorem 4.12]{Ledoux_Talagrand91}) yields
\begin{equation*}
 \EE[Z_t]\leq 4 \EE\left[\sup_{\mat{X} \in \rscsett{\beta}{r}{t}} \left|\frac{1}{n}\sum_{i=1}^{n}
 \varepsilon_i ( X_i - \tX_i) \right|\right]
 = 4 \EE\left[\sup_{\mat{X} \in \rscsett{\beta}{r}{t}} \left| \pscal{\Sigma_R}{X-\tX} \right| \right] \eqs,
\end{equation*}
where $\Sigma_R$ is defined in \eqref{eq:def:sigma_rademacher}.
Applying the duality inequality and then plugging into \eqref{proof:MassartConc} gives
\begin{equation*}
 \PP(Z_t> 4  \EE [\|\Sigma_R \|_{\sigma,\infty}] \sqrt{rt}+\gamma^2 \eta t)
 \leq \exp\left(-\eta^2nt^2/(8\gamma^4)\right)\eqs.
\end{equation*}
Since for any $a,b \in \RR  $ and $c>0$, $ab \leq (a^2/c +cb^2)/2$,
the proof is concluded by noting that, 
\begin{equation*}
4  \EE [\|\Sigma_R \|_{\sigma,\infty}] \sqrt{rt}
\leq \frac{1}{1/(2\alpha)-\eta}4\EE [\|\Sigma_R \|_{\sigma,\infty}]^2r+(1/(2\alpha)-\eta)t\eqs.
\end{equation*}
%\begin{equation*}
%  \PP(Z_t>  \frac{t}{2\alpha} + \epsilon(r,\alpha,\eta) ) \leq \exp\left(-\frac{n\eta^2t^2}{8M_\gamma^2}\right)\eqs,
%\end{equation*}
%where
%\begin{equation*}
% \epsilon(r,\alpha,\eta) \eqdef\frac{1}{1/(2\alpha)-\eta}4 \qClass^2 L_\gamma^2r\EE [\|\Sigma_R \|_{\sigma,\infty}]^2\eqs.
%\end{equation*}
 \end{proof}

 \section{Proof of Oracle inequalities and Bounds for Completion with known sampling}
 \subsection{Proof of \Cref{th:oracle_breg} }
\label{proof:oracle}
\begin{proof}
The proof is an extension (to the exponential family case) of the one proposed in
\cite[Theorem 1]{Koltchinskii_Tsybakov_Lounici11}.
For ease of notation, let us define $\gradpi \eqdef \nabla \Likpi(\tX)$
and the set $\Gamma \eqdef \{X \in \matset{m_1}{m_2}| \: \|X\|_{\infty}\leq \gamma \}$.
In view of \remarkref{rem:exp_mom}, one obtains
\begin{equation*}
  \gradpi =   \frac{\sum_{i=1}^n Y_iE_i}{n}-\nabla\Gexp^\Pi(\tX)
  =\frac{\sum_{i=1}^n (Y_iE_i - \EE[Y_iE_i])}{n} \eqs.
\end{equation*}

From the definition of $\estpi$, for any $X \in \Gamma$,
\begin{equation*}
 \Gexp^\Pi(\estpi)  - \frac{\sum_{i=1}^n \estpi_i Y_i}{n} \leq
 \Gexp^\Pi(X)  - \frac{\sum_{i=1}^n X_i Y_i}{n} + \lambda( \|X\|_{\sigma,1}-\|\estpi\|_{\sigma,1})
\end{equation*}
or equivalently
\begin{multline*}
 \Gexp^\Pi(\estpi)  - \Gexp^\Pi(\tX) -\pscal{\nabla\Gexp^\Pi(\tX)}{\estpi-\tX}  \\
 \leq
 \Gexp^\Pi(X) - \Gexp^\Pi(\tX) -\pscal{\nabla\Gexp^\Pi(\tX)}{X-\tX} 
 + \pscal{\gradpi}{\estpi-X} + \lambda( \|X\|_{\sigma,1}-\|\estpi\|_{\sigma,1})
\end{multline*}
Applying \lemmaref{lem:algebre} \eqref{ProjRel},\eqref{diffschatten} and duality yields
\begin{equation*}
 \Bregpi{\estpi}{\tX}-\Bregpi{X}{\tX} \leq \lambda (\|\estpi-X\|_{\sigma,1} +\|X\|_{\sigma,1}-\|\estpi\|_{\sigma,1})
 \leq 2 \lambda \|X\|_{\sigma,1} \eqs.
\end{equation*}
where we used the assumption $\lambda \geq \|\gradpi\|_{\sigma,\infty}$. This proves \eqref{eq:oracle_basique}.

For \eqref{eq:oracle_car}, by definition
\begin{equation*}
 \estpi= \argmin_{X \in  \matset{m_1}{m2}} F(X) 
 \eqdef   \Gexp^\Pi(X) - \frac{\sum_{i=1}^n X_i Y_i}{n}  
 + \lambda \|X\|_{\sigma,1} + \indic_\Gamma(X) \eqs,
\end{equation*}
where $\indic_\Gamma$ is the indicatrice function of the bounded closed convex set $\Gamma$ \ie
$\indic_\Gamma(x)=0$ if $x\in \Gamma$ and $\indic_\Gamma(x)= + \infty$ otherwise.
Since $F$ is convex,
$\estpi$ satisfies $0 \in \partial F(\estpi)$ with $\partial F$ denoting the subdifferential
of $F$. It is easily checked that the subdifferential $\partial \indic_\Gamma(\estpi)$
is the normal cone of $\Gamma$ at the point $\estpi$. 
Hence, $0 \in \partial F(\estpi)$ implies that
there exists $\check{V} \in \partial \| \estpi \|_{\sigma,1}$
such that
for any $X \in \Gamma$,
\begin{equation*}
 \pscal{\nabla \Gexp^\Pi(\estpi)}{\estpi-X} - \Pscal{\frac{\sum_{i=1}^n Y_iE_i}{n}}{\estpi-X}
 + \lambda \pscal{\check{V}}{\estpi-X} \leq 0 \eqs,
\end{equation*}
or equivalently
\begin{equation*}
 \pscal{\nabla \Gexp^\Pi(\estpi)-\nabla \Gexp^\Pi(\tX)}{\estpi-X} 
 + \lambda \pscal{\check{V}}{\estpi-X} \leq \pscal{\gradpi}{\estpi-X} \eqs.
\end{equation*}
For any $\tilde{x}, \bar{x}, x \in \RR$, from the Bregman
divergence definition it holds
\begin{equation}
\label{eq:magic}
 (\gexp(\check{x})-\gexp(\bar{x}))(\check{x}-x)= \Breg{G}{x}{\check{x}} + \Breg{G}{\check{x}}{\bar{x}} 
 -\Breg{G}{x}{\bar{x}} \eqs.
\end{equation}
In addition, for any $V \in \partial \|X\|_{\sigma,1}$, the subdifferential 
monotonicity  yields $\pscal{\check{V}-V}{\estpi-X} \geq 0$. Therefore
\begin{equation}
\label{eq:oracle_1}
 \Bregpi{X}{\estpi} + \Bregpi{\estpi}{\tX}-\Bregpi{X}{\tX} 
  \leq \pscal{\gradpi}{\estpi-X}  - \lambda \pscal{V}{\estpi-X} \eqs.
\end{equation}
In \cite{Watson92}, it is
shown that:
\begin{equation}
\label{eq:sub_nuc}
 \partial \|X\|_{\sigma,1}= \{ \sum_{i=1}^{r} u_iv_i^\top + \Proj_{\mat{X}}^\bot W \:\:
 | W \in \matset{m_1}{m_2}, \: \|W\|_{\sigma,\infty} \leq 1 \} \eqs,
\end{equation}
where $r \eqdef \rank(X)$, $u_i$ (\resp $v_i$) are the left (\resp right) singular
vectors of $X$ and $\Proj_{\mat{X}}^\bot$ is defined in \eqref{eq:def_projec}.
Denote by $ \mathcal{S}_1$ (\resp $\mathcal{S}_2$) the space of the left (\resp right) singular vectors
of $X$.
For $W \in \matset{m_1}{m_2}$,
\begin{equation*}
 \Pscal{\sum_{i=1}^{r} u_iv_i^\top + \Proj_{\mat{X}}^\bot W}{\estpi-X}
 =\Pscal{\sum_{i=1}^{r} u_iv_i^\top}{ P_{\mathcal{S}_1}(\estpi-X) P_{\mathcal{S}_1}}
 + \Pscal{W}{\Proj_{\mat{X}}^\bot ( \estpi) } \eqs ,
\end{equation*}
and $W$ can be chosen such that $
\pscal{W}{\Proj_{\mat{X}}^\bot ( \estpi) }=\|\Proj_{\mat{X}}^\bot ( \estpi)\|_{\sigma,1}$ and 
$\|W\|_{\sigma,\infty} \leq 1$. Taking $V \in \partial \|X\|_{\sigma,1}$ 
associated to this choice of $W$ (in the sense of \eqref{eq:sub_nuc})
 and $\|\sum_{i=1}^{r} u_iv_i^\top\|_{\sigma,\infty}=1$ yield
 \begin{multline}
 \label{eq:proof:oracle:ineq}
  \Bregpi{X}{\estpi} + \Bregpi{\estpi}{\tX}-\Bregpi{X}{\tX} 
 +\lambda \|\Proj_{\mat{X}}^\bot ( \estpi)\|_{\sigma,1} \\
 \leq \pscal{\gradpi}{\estpi-X}  + \|P_{\mathcal{S}_1}(\estpi-X) P_{\mathcal{S}_1}\|_{\sigma,1} \eqs. 
 \end{multline}
The first right hand side term can be upper bounded as follows
\begin{multline}
 \label{eq:proof:oracle:ineq2}
 \pscal{\gradpi}{\estpi-X} = \pscal{\gradpi}{\Proj_{\mat{X}}(\estpi-X)} +
 \pscal{\gradpi}{\Proj_{\mat{X}}^\bot(\estpi)} \\
 \leq \| \gradpi \|_{\sigma,\infty}( \sqrt{2 \rank(X)} \|\estpi-X\|_{\sigma,2}
 + \|\Proj_{\mat{X}}^\bot ( \estpi)\|_{\sigma,1}) \eqs,
\end{multline}
where duality and \lemmaref{lem:algebre}\speq\eqref{ProjRel} are used for the inequality.
Since $\rank(P_{\mathcal{S}_1}(\estpi-X) P_{\mathcal{S}_1})\leq \rank(X)$,
the second term satisfies
\begin{equation}
 \label{eq:proof:oracle:ineq3}
 \|P_{\mathcal{S}_1}(\estpi-X) P_{\mathcal{S}_1}\|_{\sigma,1} \leq \sqrt{\rank{X}} \|\estpi-X\|_{\sigma,2}\eqs.
\end{equation}
Using $\lambda \geq \|\gradpi\|_{\sigma, \infty}$, \eqref{eq:proof:oracle:ineq}, 
\eqref{eq:proof:oracle:ineq2} and \eqref{eq:proof:oracle:ineq3} gives
 \begin{multline}
 \label{eq:oracle_fin}
  \Bregpi{X}{\estpi} + \Bregpi{\estpi}{\tX} 
 +(\lambda-\|\gradpi\|_{\sigma, \infty}) \|\Proj_{\mat{X}}^\bot ( \estpi)\|_{\sigma,1} \\
 \leq \Bregpi{X}{\tX} + \lambda (1 + \sqrt{2})\sqrt{\rank(X)} \|\estpi-X\|_{\sigma,2}\eqs. 
 \end{multline}
 By H$\ref{A0}$ and H$\ref{A1}$, $\|\estpi-X\|_{\sigma,2} \leq \siglo^{-1}\sqrt{2m_1m_2\mu\Bregpi{X}{\estpi}}$,
 hence
 \begin{multline}
\Bregpi{\estpi}{\tX} 
+(\lambda-\|\gradpi\|_{\sigma, \infty}) \|\Proj_{\mat{X}}^\bot ( \estpi)\|_{\sigma,1} \\
 \leq \Bregpi{X}{\tX} + (\frac{1 + \sqrt{2}}{2})^2 \siglo^{-2} m_1m_2\mu\lambda^2 \rank(X) \eqs,
 \end{multline}
proving \eqref{eq:oracle_car}.
\end{proof}
% \subsection{Proof of \Cref{th:oracle_up} }
% \label{proof:oracle_up}
%  \eqref{eq:or_up_breg} is a direct consequence of \Cref{th:oracle_breg} applied 
%  with $X= \tX$ and \Cref{A4}.
%  For \eqref{eq:or_up_breg}, starting from \eqref{eq:oracle_fin}
%  with $X= \tX$ and using \Cref{A4} combined with \Cref{A0} gives
%   \begin{equation*}
%  \label{eq:oracle_fin}
%   2\| \est - \tX \|^2_{\sigma,2}
% \leq \lambda \frac{2}{\siglo^2}m_1m_2\mu(1 + \sqrt{2})\sqrt{\rank(X)} \|\estpi-X\|_{\sigma,2}\eqs,
%  \end{equation*}
%  which yields the result.
 
 \subsection{proof of \Cref{oracle:prob_up}}
 \label{subsec:proof_th_oracle_up}
\begin{proof}
 By the triangle inequality,
 \begin{equation}
 \label{eq:decomposition}
  \| \gradpi \|_{\sigma,\infty} \leq \left\| \frac{\sum_{i=1}^n (Y_i-\gexp(X_i)E_i}{n} \right\|_{\sigma,\infty}  +
  \left\| \frac{\sum_{i=1}^n \gexp(X_i)E_i}{n} - \EE[\gexp(X_1)E_1 \right\|_{\sigma,\infty} \eqs,
 \end{equation}
holds. As seen in the proof of \Cref{th:th_bis} (in \Cref{proof_bis}),
the first term of the right hand side satisfies \eqref{control:score} with probability
at least $1-d^{-1}$. If we define $Z_i=\gexp(X_i)E_i-\EE[\gexp(X_1)E_1]$,
then $\EE[Z_i]=0$ gives $\| Z_i \|_{\sigma,\infty} \leq 2 L_\gamma$,
with $L_\gamma$ defined in \eqref{eq:def_lgamma}.
A similar argument to the one used to derive \Cref{control:ezi} yields
\begin{equation*}
\norm{\EE\left[   \mat{Z_{i}}^\top\mat{Z_{i}} \right]}_{\sigma,\infty}
\leq \|\EE[ (\gexp(X_i)E_i)(\gexp(X_i)E_i)^\top] \|_{\sigma,\infty}
 \leq L^2_\gamma  \frac{1}{m} \eqs,
\end{equation*}
and the same bound holds for $\EE[Z_iZ_i^\top]$. Therefore, the uniform
version of the noncommutative Bernstein inequality (\propositionref{prop:bernstein})
ensures that with probability at least $1-d^{-1}$
\begin{equation}
\label{control:delta2}
   \left\| \frac{\sum_{i=1}^n \gexp(X_i)E_i}{n} - \EE[\gexp(X_1)E_1 \right\|_{\sigma,\infty}
   \leq c^* \max \left( \frac{L_\gamma}{\sqrt{m}} \sqrt{\frac{2\log(d)}{n}} ,
   4 L_\gamma \frac{\log(d)}{3n} \right).
\end{equation}
Combining \eqref{control:score}, \eqref{control:delta2} with the assumption made on $n$
in \Cref{oracle:prob_up}, achieves the proof.
\end{proof}

%%%%%%%%%%%%%%%%%%%%%%%%%%  UNIFORM BERNSTEIN %%%%%%%%%%%%%%%%%%%%%%%%%%%%%%
 \begin{proposition}
\label{prop:bernstein}
 Consider a finite sequence of independent random matrices 
 $(\Bern{Z_{i}})_{1 \leq i \leq n}\in \RR^{m_1 \times m_2}$ satisfying $\EE[\Bern{Z_{i}}]=0$ 
 and for some $U>0$,
 $\| \Bern{Z_{i}} \|_{\sigma,\infty} \leq U$ for all $i= 1, \dots, n$. Then for any $t>0$
 \begin{equation*}
 \PP\left( \left\|\frac{1}{n} \sum_{i=1}^{n} \Bern{Z_{i}}  \right\|_{\sigma,\infty}
 > t \right) \leq d\exp \left(- \frac{nt^2/2}{\sigma^2_Z+Ut/3} \right) \eqs,
 \end{equation*}
where $d=m_1+m_2$ and
\begin{equation*}
 \sigma^2_Z \eqdef \max \left\{\left\| \frac{1}{n} \sum_{i=1}^{n} 
 \EE[  \Bern{Z_{i}}\Bern{Z_{i}^\top} ] \right\|_{\sigma,\infty},
 \left\| \frac{1}{n}\sum_{i=1}^{n} \EE[ \Bern{Z_{i}}^\top\Bern{Z_{i}} ]\right\|_{\sigma,\infty}\right\} \eqs.
\end{equation*}
 In particular it implies that with at least probability $1-\rme^{-t}$
 \begin{equation*}
  \norm{\frac{1}{n} \sum_{i=1}^{n} \Bern{Z_{i}}}_{\sigma,\infty} \leq c^* \max \left\{ \sigma_Z \sqrt{ \frac{t + \log(d)}{n}} ,  \frac{U(t + \log(d))}{3n} \right \} \eqs,
 \end{equation*}
with $c^*=1+\sqrt{3}$.
\end{proposition}
\begin{proof}
The first claim of the proposition is Bernstein's inequality for random matrices (see for example
\cite[Theorem 1.6]{Tropp12}).
Solving the equation (in $t$) $- \frac{nt^2/2}{\sigma^2_Z+Ut/3} + \log(d)=-v$ gives with at least probability $1-\rme^{-v}$
 \begin{equation*}
  \norm{\frac{1}{n} \sum_{i=1}^{n} \Bern{Z_{i}}}_{\sigma,\infty} \leq  \frac{1}{n} \left[\frac{U}{3}(v + \log(d))+\sqrt{\frac{U^2}{9}(v + \log(d))^2+2n\sigma_Z^2(v + \log(d))}\right]\eqs,
 \end{equation*}
 we conclude the proof by distinguishing the two cases $n\sigma_Z^2 \leq (U^2/9)(v + \log(d))$ or $n\sigma_Z^2 > (U^2/9)(v + \log(d))$.
\end{proof}

\end{document}